\documentclass[11pt]{amsart} 
\renewcommand{\div}{\operatorname{div}}

\newcommand{\osc}{\operatorname{osc}}

\newcommand{\Tt}{{\mathbb{T}}}

 \newcommand{\Rr}{\mathbb R}

 \newcommand{\af}{\alpha}
 
 \newcommand{\ep}{\epsilon}
\newcommand{\be}{\beta}
 \newcommand{\ga}{\gamma}
 
 \newcommand{\de}{\delta}
\newcommand{\De}{\Delta}
 
  \newcommand{\lam}{\lambda}

\newcommand{\Ff}{\mathbb F}

\renewcommand{\div}{\operatorname{div}}

\newcommand{\tr}{\operatorname{Tr}}

\newcommand{\bx}{{\bf x}}

\newcommand{\bv}{{\bf v}}

\newtheorem{problem}{Problem}

\newtheorem{teo}{Theorem}[section]

\newtheorem{cor}{Corollary}[section]

\newtheorem{Lemma}{Lemma}[section]
\newtheorem{Proposition}{Proposition}[section]
\newtheorem{Remark}{Remark}[section]
\newtheorem{Assumption}{A}

\usepackage[dvips]{graphics}
\usepackage{graphicx}
\usepackage{color}

\begin{document}

\title{Singular mean-field games}
\author[Cirant]{Marco Cirant}
\address{M. Cirant, Dipartimento di Matematica, Universit\`a di Padova, Via Trieste 63, 35121, Padova, Italy.}
\email{cirant@math.unipd.it}
\author[Gomes]{Diogo A. Gomes}
\address{D. Gomes, King Abdullah University of Science and Technology (KAUST), CEMSE Division, Thuwal 23955-6900, Saudi Arabia.}
\email{diogo.gomes@kaust.edu.sa}
\author[Pimentel]{Edgard A. Pimentel}
\address{E. Pimentel, Department of Mathematics, Universidade Federal de S\~ao Carlos, 13506-905, S\~ao Carlos, Brazil.}
\email{edgard@dm.ufscar.br}
\author[S\'anchez-Morgado]{H\'ector S\'anchez-Morgado } 
\address{H. S\'anchez-Morgado, Instituto de Matem\'aticas, Universidad Nacional Aut\'onoma de M\'exico. 
Cd. M\'exico 04510. M\'exico.} 
\email{hector@matem.unam.mx}

\thanks{M. Cirant is partially supported by the Fondazione CaRiPaRo Project "Nonlinear Partial Differential Equations: Asymptotic Problems and Mean-Field Games" and the INdAM-GNAMPA project "Fenomeni di segregazione in sistemi stazionari di tipo Mean Field Games a pi\`u popolazioni".	
}

\thanks{D. Gomes was partially supported by KAUST baseline and start-up funds.
}

\thanks{E. Pimentel was partially supported by FAPESP (Grant 2015/13011-6) and PPGM-UFSCar baseline funds.
}
        
\date{\today}

\maketitle

\begin{abstract}

\noindent Here, we prove the existence of smooth solutions for mean-field games with a singular mean-field coupling; that is,  a coupling in the Hamilton-Jacobi equation of the form $g(m)=-m^{-\af}$. We consider stationary and time-dependent settings. The function $g$  is monotone, but it is not bounded from below. With the exception of the logarithmic coupling, this is the first time that MFGs whose coupling is not bounded from below is examined in the literature.  This coupling arises in models where agents have a strong preference for low-density regions. Paradoxically, this causes the agents to spread and prevents the creation of solutions with a very-low density. To prove the existence of solutions, we consider an approximate problem for which the existence of smooth solutions is known. Then, we prove new a priori bounds for the solutions that show that $\frac 1 m$ is bounded. Finally, using a limiting argument, we obtain the existence of solutions. The proof in the stationary case relies on a blow-up argument and in the time-dependent case on new bounds for $m^{-1}$. 
\medskip


\medskip

\end{abstract}

\section{Introduction}\label{sec int}

Here,  we examine singular second-order reduced mean-field games (MFGs)  both in the stationary and time-dependent settings.
Our goal is to determine conditions that ensure the existence of classical solutions of the following problems: 

\begin{problem}[Stationary]
\label{P1sta}
Given $H:\Tt^d\times \Rr^d\to \Rr$, $H\in C^\infty$ and $\alpha>0$,  find $u,m\in C^\infty(\Tt^d)$, $m>0$, and $\overline{H} \in \Rr$ that solve
 the (stationary) MFG:
\begin{equation}\label{eq:mfgsta}
\begin{cases}
        -\Delta u(x)\,+\,H(x,Du(x))\,=\,\overline{H}\,-\,m^{-\af}(x)&\mbox{  in  }\Tt^d\\

        -\De m(x)-\div(D_pH(x,Du)m(x))=0&\mbox{ in }\Tt^d,
\end{cases}
\end{equation}
satisfying the following condition:
\begin{equation}\label{eq:itbcsta}
        \int_{\Tt^d}m(x)dx\,=\,1.
\end{equation} 
\end{problem}

\begin{problem}[Time-dependent]
\label{P1}
Given $H:\Tt^d\times \Rr^d\to \Rr$, $H\in C^\infty$, $\alpha>0$, and  $u^T, m^0 \in C^\infty(\Tt)$, $m^0>0$,  find $u,m\in C^\infty(\Tt^d\times [0,T])$, $m>0$, that solve
 the MFG:
\begin{equation}\label{eq:mfg}
\begin{cases}
-u_t(x,t)\,+\,H(x,Du(x,t))\,=\,\De u(x,t)\,-\,m^{-\af}(x,t)&\mbox{  in  }\Tt^d\times[0,T]\\

m_t(x,t)-\div(D_pH(x,Du)m(x,t))=\De m(x,t)&\mbox{ in }\Tt^d\times[0,T]
\end{cases}
\end{equation}
and satisfy  the  initial-terminal boundary conditions:
\begin{equation}\label{eq:itbc}
\begin{cases}
u(x,T)\,=\,u^T(x)&\mbox{ in }\Tt^d\\

m(x,0)\,=\,m^0(x)&\mbox{ in }\Tt^d.
\end{cases}
\end{equation} 
\end{problem}

The MFG\ \eqref{eq:mfg}-\eqref{eq:itbc} arises as a model for the interaction of a large number of rational agents that strongly prefer empty or low-density regions. More precisely, we let $W_s$ be a $d$-dimensional Brownian motion in a filtered probability
space $(\Omega, \Ff, P)$ and  let  $\mathcal{V}$ be a set of admissible
controls; that is,
bounded progressively measurable processes on $[t, T]$ with values in $\Rr^d$.
 Each agent seeks to optimize a  stochastic control problem with dynamics
\begin{align}
\label{dyn}
&d\bx=\bv ds+\sqrt{2} dW_s,\\\notag  &\bx(t)=x,  
\end{align} 
with $\bv\in \mathcal{V}$, 
and cost functional
\begin{equation}\label{soc}
J(x,t;\bv)= \mathbb{E}^{(x,t)}\left[\int_t^T L(\bx(s),\bv(s))-\frac
{ds}{(m(\bx(s),s))^\af} + u^T(\bx(T))\right]. 
\end{equation}
In the above cost,  the Lagrangian, $L,$ is given by 
\begin{equation}
  \label{hatele}
  L(x,v)\,:=\,\sup_{p\in\mathbb{R}^d}\left[ -p\,\cdot\,v\,-\,H(x,p)\right.].
\end{equation}
Under suitable regularity conditions for $m$, it is well known that  
\[
        u(x,t)\,=\,\inf_{\bv\in\mathcal{V}}J(\bx,t;\bv)
\]
solves the first equation in \eqref{eq:mfg}. Moreover, again under suitable regularity conditions, the optimal control 
is given in feedback form as
\[
\bv(s)=-D_pH(\bx(s), D_xu(\bx(s),s)).
\]
This, in turn, gives that the law, $m$,  of the diffusion \eqref{dyn} solves
the second equation in \eqref{eq:mfg}.

The stationary problem, Problem \ref{P1sta}, has a similar interpretation; it describes the ergodic problem associated with the ergodic cost
\begin{equation}\label{socsta}
J(x;\bv)= \lim_{T\to\infty}\frac{1}{T}\mathbb{E}^{(x,0)}\left[\int_0^T L(\bx(s),\bv(s))-\frac
{ds}{(m(\bx(s),s))^\af} \right].
\end{equation}
A solution $(u,\,\overline{H},\,m)$ to Problem \ref{P1sta} also describes the equilibrium configuration of the system, due to the ergodic (limiting) behavior encoded in \eqref{socsta}.

To investigate Problems \ref{P1sta} and \ref{P1}, we introduce two auxiliary, regularized MFGs. In the stationary setting, we examine the problem:

\begin{problem}[Stationary]
\label{P2sta}
Given $H:\Tt^d\times \Rr^d\to \Rr$, , $H\in C^\infty$, $\alpha>0$ and $\epsilon>0$  find $u^\epsilon,m^\epsilon\in C^\infty(\Tt^d)$, with $m^\epsilon>0$, and $\overline{H}^\epsilon$ solving
 the MFG:
\begin{equation}\label{mfgregsta}
\begin{cases}
- \Delta u^\epsilon + H(x, D u^\epsilon) =  \overline{H}^\epsilon - (m^\epsilon + \epsilon)^{-\alpha} &\;\;\;\;\;\mbox{in}\;\;\;\;\;\Tt^d\\
- \Delta m^\epsilon -{\rm div}(D_p H(x, D u^\epsilon) \, m^\epsilon) =  0& \;\;\;\;\;\mbox{in}\;\;\;\;\;\Tt^d,
\end{cases}
\end{equation}
satisfying the following condition:
\begin{equation}\label{eq:itbcregsta}
        \int_{\Tt^d}m^\epsilon(x)dx\,=\,1.
\end{equation} 
\end{problem}

In the time-dependent setting, we study the following problem:

\begin{problem}[Time-dependent]
\label{P2}
Given $H:\Tt^d\times \Rr^d\to \Rr$, , $H\in C^\infty$, $\alpha>0$, $\epsilon>0$, and  $u^T, m^0 \in C^\infty(\Tt)$,
$m^0>0$,  find $u^\epsilon,m^\epsilon\in C^\infty(\Tt^d\times [0,T])$, $m^\epsilon>0$, that solve
 the MFG:
\begin{equation}\label{eq:mfgreg}
        \begin{cases}
                -u^\ep_t(x,t)+H(x,Du^\ep(x,t))=\De u^\ep(x,t)-(m^\ep(x,t)+\ep)^{-\af}&\mbox{
in }\Tt^d\times[0,T]\\

                m^\ep_t(x,t)-\div(D_pH(x,Du^\ep)m(x,t))=\De m^\ep(x,t)&\mbox{
in }\Tt^d\times[0,T],
        \end{cases}
\end{equation}
where
\begin{equation}\label{eq:itbcreg}
        \begin{cases}
                u^\ep(x,T)\,=\,u^T(x)&\mbox{ in }\Tt^d\\

                m^\ep(x,0)\,=\,m^0(x)&\mbox{ in }\Tt^d.
        \end{cases}
\end{equation}
\end{problem}

The existence theory for MFGs has seen significant advances in recent years. A typical 
MFG with a local coupling is
\[
\begin{cases}
-u_t+\frac{|Du|^2}{2}=\Delta u +g(m). \\
m_t-\div(m Du)=\Delta m. 
\end{cases}
\]
In \cite{ll1, ll2,ll3}, the authors proved several a priori estimates for weak solutions of MFGs. Because of the quadratic structure, the preceding case was one of the first MFGs to be studied. For $g$ increasing and bounded from below,  the proof of the existence of a solution was outlined in \cite{lcdf} and detailed in  \cite{CLLP} using the Hopf-Cole transformation (also see \cite{gueant3} and   \cite{MR2974160}).  General parabolic MFGs with a nonlocal coupling were addressed in \cite{lcdf} (see also \cite{cardaliaguet}). The existence of weak solutions in the case of local couplings bounded by below and of polynomial growth in $g$\ was considered in \cite{cgbt}, in the variational setting, and in \cite{porretta2}, using general partial differential equations (PDE)\ methods (see  \cite{porretta} for the planning problem). Finally, the parabolic case with a polynomial nonlinearity $g$ was examined in \cite{GPM2} for subquadratic Hamiltonians and in \cite{GPM3} for superquadratic Hamiltonians. In all these cases, the mean-field couplings were non-singular for $m\geq 0$\ and bounded from below. Explicit solutions for MFGs have been studied in \cite{GNP1} and \cite{GNP2}. The study of congestion problems provides examples of singular MFGs. A model problem is
\[
\begin{cases}
-u_t+\frac{|Du|^2}{2 m^\alpha}=\Delta u, 
m_t-\div(m^{1-\alpha} Du)=\Delta m
\end{cases}
\]
with $\alpha>0$. In the stationary case, the congestion problem was investigated in \cite{GMit}, and in the time-dependent case
in \cite{GVrt2} and \cite{Graber2} for a small terminal time, $T$. For
the general case of stationary MFGs in the presence of congestion, we refer the reader to \cite{EvangelistaGomes}.

In the problems discussed above, the Lagrangian corresponding to the Hamilton-Jacobi equation is bounded from below. Thus, it is relatively easy to get lower bounds for $u$ using the maximum principle or, alternatively, a control theory interpretation of the problem. Bounds for $u$ in $L^\infty$ are essential to prove the regularity of solutions. Thus,  if the coupling is not bounded from below, the regularity problem poses substantial challenges. Prior to this paper, the only known existence results for MFGs with singularities that are not bounded from below were those for a logarithmic nonlinearity, $g(m)=\ln m$. This case was examined in the stationary setting in  \cite{GPatVrt}, \cite{GM}, and \cite{PV15},  in the time-dependent setting in \cite{GPim2}, and in a specific case in \cite{gueant3}. Additional examples were discussed in \cite{AFG}, \cite{RFG2}, \cite{GNP1}, \cite{GNP2}, and \cite{TheBook}.
However, the techniques used to study logarithmic nonlinearities do not apply to power-like nonlinearities as the ones in \eqref{eq:mfg}, except in  the one-dimensional case, where the methods in \cite{GNP1} and \cite{GNP2} can be adapted accordingly.  

In addition to its interest for the MFG theory, models similar to \eqref{eq:mfg} appear in  cavitation,  free boundary,
and obstacle problems. A typical problem is
\begin{equation}\label{mildell}
        \De v(x) \,=\, v^{\sigma-1}\mbox{ in } B_1,
\end{equation}
where $\sigma\in(0,1)$. For results in this direction, we refer the reader to \cite{altphilips}, \cite{mont1}, \cite{mont2}, and \cite{delpino}.

In Section \ref{sec
maop}, we discuss the main technical assumptions under which our results are valid (Assumptions A\ref{sub1}-A\ref{alphastat}). 
 For example,  we observe that the following Hamiltonian
satisfies these assumptions:
\[
        H(x,p)\,=\,a(x)\left(1\,+\,|p|^2\right)^\frac{\gamma}{2}\,+\,V(x),
\]
where  $a,\,V\,\in\,\mathcal{C}^\infty(\Tt^d)$, with $a>0$ and $\gamma$ in a suitable range to be discussed later (see Section \ref{sec maop}). Our first main result is the following theorem:

\begin{teo}[Stationary case]\label{thm_stationary}
Suppose Assumptions A\ref{sub1}-A\ref{trace} and A\ref{alphastat} hold (see Section
\ref{sec maop}). Then, there exists a unique solution $(u,\overline{H}, m)$ of Problem \ref{P1sta}. 
\end{teo}

To handle the singular coupling in the stationary case, we study Problem \ref{P2sta} and use a blow-up method to obtain that $m^\epsilon$ is bounded from below by a constant independent on $\epsilon$. Our method is related to the one in \cite{marcoCPDE}, in the setting of \textit{focusing} MFG, where the coupling is a function, which is unbounded from below as $m \to \infty$: in that case, it was shown that some integrability of $m$ is enough to obtain boundedness of $m$ in $L^\infty$. Here, a similar procedure applies: the integrability of $m^{-1}$ implies bounds from above for $m^{-1}$, which allows passing to the limit in Problem \ref{P2sta}. Second-order estimates that use the monotonicity of $-m^{-\alpha}$ give that  $(m^\epsilon)^{-1}$ is bounded in some $L^p$ space. In Section \ref{elemeststa}, we establish preliminary estimates. Then, in Section \ref{sec_proofsta}, we conclude the proof of Theorem \ref{thm_stationary}.
\vspace{.2in}

Our second main result concerns the time-dependent setting:

\begin{teo}[Time-dependent case]\label{mainteo}
        Suppose Assumptions A\ref{sub1}-A\ref{sub2} and A\ref{gamma} hold (see Section
\ref{sec maop}). Then, there exists a unique solution $(u,m)$ to Problem \ref{P1}. 
\end{teo}

To prove Theorem \ref{mainteo}, we use a limiting argument. For that,  we rely on the properties of solutions to Problem \ref{P2}.  The existence of solutions to Problem \ref{P2}
follows from standard arguments, see, for example, \cite{GPM2}
or \cite{GPM3}. Here, we  obtain a priori estimates for the solutions $(u^\epsilon,m^\epsilon)$
to Problem \ref{P2} that are uniform in $\epsilon$. In Section \ref{hle}, we prove
our main estimate, which gives bounds for the coupling in Problem \ref{P2},
 $-(m^\epsilon+\epsilon)^{-\alpha}$,
that are uniform in $\epsilon$. Subsequently, in Sections \ref{rfp} and \ref{sec:reg-adj},
we examine the Fokker-Planck equation and the Hamilton-Jacobi equation in
\eqref{eq:mfgreg}. As remarked previously, a critical point in our estimates for the Hamilton-Jacobi equation
is that the coupling in Problem \ref{P2} is not bounded from below uniformly
in $\epsilon$. Thus, to obtain bounds for $u^\epsilon$ is $L^\infty$, we
use a delicate argument based on the nonlinear adjoint method introduced in \cite{E3} (also see \cite{T1}). Finally, in Section \ref{liphj}, we give uniform Lipschitz bounds for $u^\epsilon$. These
estimates ensure the compactness of the solution $(u^\epsilon,m^\epsilon)$.
Thus, by extracting a subsequence, we have\[
        (u^\epsilon,m^\epsilon)\,\to\,(u,m),
\]
as $\epsilon\to 0$, for some $(u,m)$ that solves
\eqref{eq:mfg}-\eqref{eq:itbc} in the weak sense
and inherits the regularity of $(u^\epsilon,m^\epsilon)  $. In Section \ref{liphj},
we conclude the Theorem \ref{mainteo} along these lines. 

\noindent {\bf Acknowledgments} The results in this paper were partially obtained during scientific visits of the authors to PPGM-UFSCar and IMPA (Brazil); the authors are grateful for the kind hospitality received from these Institutions.

\section{Main assumptions}\label{sec maop}

Here, we discuss the main assumptions used in this paper. Our first assumption,
stated next, requires that the Hamiltonian satisfies standard growth and coercivity properties; these are similar to the ones in the literature. 

\begin{Assumption}\label{sub1}The Hamiltonian $H:\Tt^d\times\Rr^d\to\Rr^d$ is smooth. Also, for fixed $x$, $p\mapsto H(x,p)$ is a strictly convex function. In addition,  
                there are constants, $C_1,\,C_2>0,$ such that
                \[
        \frac{1}{C_2}|p|^\ga\,-\,C_1\,\le\,     H(x,p)\,\le\,C_1\,+\,C_2|p|^\ga,
                \]
and, without loss of generality, we suppose that $H(x,p)\ge 0$. 
In addition, we have
                \[
                        \left|D_pH(x,p)\right|\,\le\,C_1\,+\,C_2|p|^{\ga-1}
                \]
       and 
                \[
                        \left|D_xH(x,p)\right|\,\le\,C_1\,+\,C_2 H(x,p).
                \]
        
\end{Assumption}

The next assumption requires a lower bound for the Lagrangian given by \eqref{hatele}.  We recall that  if $v=-D_pH(x,p)$, then
\[
L(x,v)=D_pH(x,p)\cdot p\,-\,H(x,p)\,.
\]

\begin{Assumption}\label{sub2}
        There are constants, $C_1,\,C_2>0,$ such that
        \[
            D_pH(x,p)\cdot p\,-\,H(x,p)\,\ge\,\frac{H(x,p)}{C_2}\,-\,C_1.
        \]
\end{Assumption}

Our next assumption plays a major role in the second-order estimates for the stationary problem.

\begin{Assumption}\label{trace}
For every $\delta\in(0,1)$, there exists $C_\delta>0$ so that
\[
        \tr(D^2_{xp}H(x,p)M)\,\leq\,\delta\tr(D^2_{pp}H(x,p)M^2)\,+\,C_\delta H(x,p),
\]
for any matrix $M$.
\end{Assumption}

For the time-dependent problem, our main technical condition concerns the growth of the Hamiltonian.
\begin{Assumption}\label{gamma}
We assume the Hamiltonian's growth parameter, $\gamma$, satisfies
\[
        1\,<\,\ga\,<\,\frac{d+2}{d+1}.
\]
\end{Assumption}
The preceding assumption is used together with the non-linear adjoint method to get bounds for $u$\ in $L^\infty$. The main reason for this constraint is that the non-linearity is not bounded from below. Thus, proving lower bounds for $u$, which is usually  an immediate corollary of the maximum principle or of the optimal control representation, becomes a non-trivial task.

For the stationary problem, we need the following assumption:

\begin{Assumption}\label{alphastat}
We assume that $\alpha > \bar{\alpha}_{d,\gamma}$, where
\[
\bar{\alpha}_{d,\gamma} = 
\begin{cases}
        0 &     \text{if $\gamma < \frac{d}{d-1}$}, \\
        1 &     \text{if $\gamma \ge 2$ and $d = 2$}, \\
        \max\left(\frac{\gamma}{\gamma(3-d) + d-2},1\right) &   \text{if $\frac{d}{d-1} \le \gamma < \frac{d-2}{d-3}$ and $d \ge 3$}.
\end{cases}
\]
\end{Assumption}
Note that $\bar{\alpha}_{N,\gamma} \to \infty$ as $\gamma \to ((d-2)/(d-3))^-$. In the next Section, we examine the stationary case and prove Theorem \ref{thm_stationary}.

\section{The stationary case}\label{sec_stationary}

The proof of Theorem \ref{thm_stationary} relies on a limiting argument. In turn, it depends on a priori estimates for the solutions of Problem \ref{P2sta}, uniform in $\epsilon$. We produce these estimates by using a blow-up method, see \cite{marcoCPDE}. Preliminary results, central to our argument, are developed in what follows.

\subsection{Elementary estimates}\label{elemeststa}

We start by proving a first-order estimate for solutions of Problem \ref{P2}.

\begin{Lemma}[First-order estimates]\label{lem_foesta}Let $(u^\epsilon,\overline{H}^\epsilon, m^\epsilon)$ be a solution to \eqref{mfgregsta}. Suppose Assumptions A\ref{sub1}-A\ref{sub2} hold. Then, there exists a constant, $C > 0,$ such that
\begin{equation}\label{stat_first_ord}
\int_{\Tt^d} \frac{1}{(m^\epsilon + \epsilon)^\alpha} \, dx +  \int_{\Tt^d} H(x, D u^\epsilon)  \, dx + \int_{\Tt^d} H(x, D u^\epsilon) m^\epsilon  \, dx + |\bar{H}^\epsilon| \le C.
\end{equation}
\end{Lemma}

\begin{proof}

For ease of presentation, we drop the superscript $\epsilon$ in what follows. We integrate the Hamilton-Jacobi equation in \eqref{mfgregsta} to
 obtain
\begin{equation}\label{foe0}
        -C\,\leq \,\int_{\mathbb{T}^d}H(x,Du)dx\,=\,\overline{H}\,-\,\int_{\mathbb{T}^d}\frac{dx}{(m\,+\,\epsilon)^\alpha}\,\leq\,\overline{H}.
\end{equation}
The former inequality gives a lower bound for the ergodic constant, $\overline H$. Next, we fix a number $\delta\ll1$ and multiply the first equation in \eqref{mfgregsta} by $m\,-\,\delta$. Then, we multiply the second equation in \eqref{mfgregsta} by $-u$ and sum them. Integrating by parts, we get
\[
        \int_{\mathbb{T}^d}(H-D_pH\cdot Du)m -\delta\int_{\mathbb{T}^d}Hdx\,=\,(1-\delta)\overline{H}\,-\,\int_{\mathbb{T}^d}\frac{m-\delta}{(m+\epsilon)^\alpha}dx.
\]
By rearranging the terms in the previous equality, we have
\begin{equation}\label{foe1}
        (1-\delta)\overline{H}\,+\, \frac{1}{C_2}\int_{\mathbb{T}^d}|Du|^\gamma m dx\,+\,\delta\int_{\mathbb{T}^d}H dx\,\leq\, \int_{\mathbb{T}^d}\frac{m-\delta}{(m+\epsilon)^\alpha}dx + C_1.
\end{equation}
As before, we use the convexity of
\[
        z\,\mapsto\,\frac{z^{1-\alpha}}{\alpha-1}
\]
to conclude that
\[
        \int_{\mathbb{T}^d}\frac{m-\delta}{(m+\epsilon)^\alpha}dx\,\leq\, \frac{(\delta\,+\,\epsilon)^{1-\alpha}}{\alpha\,-\,1}\,- \frac{1}{\alpha\,-\,1} \,\int_{\mathbb{T}^d}\frac{dx}{(m\,+\,\epsilon)^{\alpha-1}}.
\]
The former inequality, together with \eqref{foe1}, implies there exists a constant, $C>0$, such that
\[
        \overline{H}\,+\,\int_{\mathbb{T}^d}|Du|^\gamma m \, dx\,+\,\int_{\mathbb{T}^d}H dx\,+\,\int_{\mathbb{T}^d}\frac{dx}{(m+\epsilon)^{\alpha-1}}dx\,\leq\, C.
\]
Combined with \eqref{foe0}, the prior inequality concludes the proof.
\end{proof}

\begin{Lemma}[Second-order estimates]\label{lem_soesta}Let $(u^\epsilon,\overline{H}^\epsilon, m^\epsilon)$ be a solution to \eqref{mfgregsta}. Suppose Assumptions A\ref{sub1}-A\ref{sub2} and A\ref{trace} hold. Then,
\begin{equation}\label{stat_sec_ord}
        \int_{\mathbb{T}^d}\frac{|Dm|^2}{(m+\epsilon)^{\alpha+1}}dx\,+\,\int_{\mathbb{T}^d}\tr\left(D^2_{pp}H(D^2u)^2\right)mdx\,\leq\,C
\end{equation}
for some positive constant $C$.
\end{Lemma}

\begin{proof}
As before, we omit the superscript $\epsilon$ throughout the proof. First, we apply the Laplace operator to the first equation in \eqref{mfgregsta} to get
\begin{align}\label{soesta1}
        \Delta\Delta u+\tr\left(D^2_{pp}H(D^2u)^2\right)+\Delta H+2\tr\left(D^2_{px}H D^2u\right)+D_pHD(\Delta u)=\div\left(\frac{\alpha Dm}{(m+\epsilon)^{\alpha+1}}\right).
\end{align}
By multiplying \eqref{soesta1} by $m$, integrating by parts, and using the second equation in \eqref{mfgregsta}, we obtain
\begin{align*}
        \alpha\int_{\mathbb{T}^d}\frac{|Dm|^2}{(m+\epsilon)^{\alpha+1}}dx+&\int_{\mathbb{T}^d}\tr\left(D^2_{pp}H(D^2u)^2\right)mdx=-\int_{\mathbb{T}^d}(\Delta H+2\tr(D^2_{px}HD^2u))mdx\\
        &\quad\leq \int_{\mathbb{T}^d}\left(|D^2_{xx}H|+\delta\tr(D^2_{pp}H(D^2u)^2)+C_\delta H\right)mdx.
\end{align*}
By choosing $\delta\in(0,1)$ and using the estimates in Lemma \ref{lem_foesta}, we conclude
\[
        \int_{\mathbb{T}^d}\frac{|Dm|^2}{(m+\epsilon)^{\alpha+1}}dx\,+\,\int_{\mathbb{T}^d}\tr\left(D^2_{pp}H(D^2u)^2\right)mdx\,\leq\,C
\]
for some constant $C>0$, which finishes the proof.
\end{proof}

\subsection{The blow-up and Proof of Theorem \ref{thm_stationary}}\label{sec_proofsta}

Under suitable assumptions, the following proposition states that integrability of $(m^\epsilon + \epsilon)^{-1}$ is sufficient to obtain bounds from below for $m^\epsilon + \epsilon$.

\begin{Proposition}\label{lem_infty}Suppose Assumptions A\ref{sub1}-A\ref{sub2} hold. Fix $K>0$ and assume
\begin{equation}\label{p_ass}
p >  \frac{\alpha d}{\gamma'}.
\end{equation}
Then, there exists $C>0$, independent on $\epsilon$, such that for any solution $(u^\epsilon,\overline{H}^\epsilon, m^\epsilon)$ to Problem \ref{P2sta} satisfying
\[
\int_{\Tt^d} \frac{1}{(m^\epsilon + \epsilon)^p} \, dx \le K,
\]
we have
\begin{equation}\label{inftybound}
m^\epsilon + \epsilon \ge C
\end{equation}
for all $\epsilon$.
\end{Proposition}
\begin{proof} We argue by contradiction. Suppose the claim in the proposition is false. Then, there are $\eta^\epsilon > 0, x^\epsilon \in \Tt^d$, such that
\[
0 < \eta^\epsilon := m^\epsilon(x^\epsilon) + \epsilon = \min_{\Tt^d} (m^\epsilon + \epsilon) \to 0, \quad \text{as $\epsilon \to 0$}.
\]
For ease of presentation, we proceed with several steps.

{\bf Step 1.} We define
$a^\epsilon := (\eta^\epsilon)^{\alpha/\gamma'}$ 
and set $\mathcal{T}^\epsilon = \{x: x^\epsilon + a^\epsilon x \in \Tt^d\}$.
Next, we consider
the following blow-up sequences: 
\begin{equation}\label{blowupdef}
v^\epsilon(x) := (a^\epsilon)^{\gamma'-2} u^\epsilon(x^\epsilon + a^\epsilon x), \quad \mu^\epsilon(x) := \frac{m^\epsilon(x^\epsilon + a^\epsilon x) + \epsilon}{\eta^\epsilon} \quad 
\end{equation}
for $x \in \mathcal{T}^\epsilon $.
To simplify the notation and because no ambiguity occurs,
in the remaining part of this proof, we omit 
the superscript $\epsilon$ 
in $a^\epsilon$, $\mu^\epsilon$, and $v^\epsilon$. 

Then, $(v, \mu)$ solves
\begin{equation}\label{MFGn}
\begin{cases}
- \Delta v + H^\epsilon(x, D v) = a^{\gamma'} \bar{H} - a^{\gamma'}(\eta \mu)^{-\alpha} \\
- \Delta \mu -{\rm div}(D_p H^\epsilon(x, D v) \, \mu) + \frac{\epsilon}{\eta}{\rm div}(D_p H^\epsilon(x, D v)) =  0& \text{in $\mathcal{T}^\epsilon$}, \\
\end{cases}
\end{equation}
where $H^\epsilon(x,p) = a^{\gamma'}H(x^\epsilon + a^\epsilon x, a^{1-\gamma'}p)$. Because $a \to 0$, $H^\epsilon$ satisfies Assumptions A1-A2, where $C_1$, $C_2$, $\gamma$ are the same as for $H$. In particular,
\begin{equation}\label{Heps_bound}
|D_p H^\epsilon(x,p)| \le a^{\gamma'}C_1 + {C_2}|p|^{\gamma-1} .
\end{equation}
Moreover, $\mu(0) = 1$ and $\mu \ge 1$ in $\mathcal{T}^\epsilon$, so 
\begin{equation}\label{Fbound}
0 \le a^{\gamma'}(\eta \mu)^{-\alpha} \le 1 \quad \text{on $Q^\epsilon$},
\end{equation}
due to the choice of $a$.

{\bf Step 2.} The gradient of $v$ is bounded in $L^r(B_2(0))$ for any $r>1$. It suffices to observe that
\[
| - \Delta v + H^\epsilon(x, D v) | \le 1 + a^{\gamma'} |\bar{H}| \le 2 \;\;\;\;\;\mbox{in} \;\;\;\;\;\mathcal{T}^\epsilon
\]
by \eqref{Fbound} and \eqref{stat_first_ord} when $\epsilon$ is small enough. Therefore, we apply the integral Bernstein estimates for viscous Hamilton-Jacobi-Bellman (HJB) equations to get
\begin{equation}\label{Lrbound}
\| D v \|_{L^r(B_2(0))} \le C_r.
\end{equation}

For the integral Bernstein method in the special case $H(x,p) = |p|^\gamma$, see \cite{lasry1989nonlinear}. For the general case, we argue as in \cite{cirantjmpa} or \cite{PV15}.

Moreover, from elliptic regularity, we gather that 
\begin{equation}
\label{2derv}
\|D^2v\|_{L^q(B_2(0))}\leq C_q, 
\end{equation}
for any $q>1$. 

{\bf Step 3.} Now, we show that $\mu$ is uniformly bounded from above in the ball $B_1(0)\cap\mathcal{T}^\epsilon$. Indeed,
according to \eqref{Lrbound} and \eqref{2derv}, 
for any $q>1$, 
$B(x) = -D_p H^\epsilon(x, D v(x))$ satisfies
\[
\| B \|_{W^{1,q}(B_2(0))} \le C_q. 
\]
Moreover, $\mu$ is a positive solution to the following equation in divergence form
\[
-{\rm div}\left(D \mu + B \, \mu - \frac{\epsilon}{\eta} B \right) =  0 \;\;\;\;\;\mbox{in}\;\;\;\;\;\mathcal{T}^\epsilon.
\]

By the definition of $\eta$, we have $\eta \ge \epsilon$.  So,  $0 \le \epsilon / \eta \le 1$. Thus, we apply Harnack's inequality (see, for example, \cite{Serrin}) to get
\[
\max_{B_1(0)} \mu \le C ( \min_{B_1(0)} \mu + k)  = C (\mu (0)  + k )= C(1 + k)
\]
for some positive constants $C,\,k$ depending only on the upper bound of $\| B \|_{L^q(B_2(0))}$. 

Hence, for all $p > 0$, we have
\begin{equation}\label{belowLalpha}
\int_{B_1(0)} \left(\frac{1}{\mu}\right)^p \, dx \ge \delta_p > 0.
\end{equation}

{\bf Step 4.} Finally, evaluating the integral of $(1/\mu)^p$ on $\mathcal{T}^\epsilon$ gives
\begin{equation}\label{capitalA}
\int_{\mathcal{T}^\epsilon} \left(\frac{1}{\mu}\right)^p \, dx= \frac{\eta^p}{a^{d}} \int_{\Tt^d} \left(\frac{1}{m + \epsilon}\right)^p \, dx=  \eta^{p - d \frac{\alpha}{\gamma'}}\int_{\Tt^d} \frac{1}{(m + \epsilon)^p} \, dx \to 0
\end{equation}
as $\epsilon \to 0$, in view of the assumptions of the proposition. However, \eqref{capitalA} contradicts \eqref{belowLalpha}.

\end{proof}

First, we examine the case $\gamma < d/(d-1)$, where no additional assumptions on $\alpha$ are needed.

\begin{cor}\label{ddm1} Suppose Assumptions A\ref{sub1}-A\ref{sub2} hold and that $\gamma < d/(d-1)$. Then, for any solution $(u^\epsilon,\overline{H}^\epsilon, m^\epsilon)$ to Problem \ref{P2sta}, there exists $C > 0$, not depending on $\epsilon$, such that
\begin{equation}\label{inftybound}
m^\epsilon + \epsilon \ge C
\end{equation}
for all $\epsilon$.
\end{cor}

\begin{proof} It suffices to apply Proposition \ref{lem_infty} with $p = \alpha$ and the estimate \eqref{stat_first_ord}. Note that \eqref{p_ass} reads $\gamma' > d$.
\end{proof}

Next, we consider case when $\gamma$ is above $d/(d-1)$, which requires $\alpha$ not be too small.

\begin{cor}\label{dm2dm3} Suppose Assumptions A\ref{sub1}-A\ref{sub2} hold, $d \ge 3$, $\gamma < (d-2)/(d-3)$, and
\begin{equation}\label{baralpha}
\alpha > \max\left(\frac{\gamma}{\gamma(3-d) + d-2},1\right).
\end{equation}
Then, there exists a constant, $C > 0$, independent of $\epsilon,$ such that for any solution $(u^\epsilon,\overline{H}^\epsilon, m^\epsilon)$ to Problem \eqref{P2sta}, we have
\begin{equation}\label{inftybound}
m^\epsilon + \epsilon \ge C
\end{equation}
for all $\epsilon$.
\end{cor}

\begin{proof} If $d>2$, $W^{1,2}(\Tt^d)$ is continuously embedded into $L^{\frac{2d}{d-2}}(\Tt^d)$; then, we have
\[
\left( \int_{\Tt^d} (m + \epsilon)^{(1-\alpha)\frac{d}{d-2}} \, dx \right)^{\frac{d-2}{d}} \le C \left( \int_{\Tt^d}\left |D m^{\frac{1-\alpha}{2}} \right|^2 \, dx + \int_{\Tt^d} (m+\epsilon)^{{1-\alpha}} \, dx \right).
\]
The right-hand side of the previous inequality is bounded in view of \eqref{stat_first_ord} and  \eqref{stat_sec_ord} because
\[
        \left |D m^{\frac{1-\alpha}{2}} \right|^2 = \left (\frac{1-\alpha}{2}\right)^2 m^{-(\alpha+1)}|D m|^2.
\]
Finally, because \eqref{baralpha} guarantees \eqref{p_ass},  we apply Proposition \ref{lem_infty} with $p = (\alpha - 1)\frac{d}{d-2}$. 
\end{proof}

In the case $d=2$, arguing as in Corollary \ref{dm2dm3} 
gives that
\[
\int_{\Tt^d} (m + \epsilon)^{-p} \, dx \le C
\]
for all $p > 0$, provided
 $\alpha > 1$ and the conditions in Proposition \ref{lem_infty} hold. Finally, we prove our main result in the stationary case, Theorem \ref{thm_stationary}.

\begin{proof}[Proof of Theorem \ref{thm_stationary}] Depending on the values of $\gamma$ and $\alpha$, Corollaries \ref{ddm1}-\ref{dm2dm3} provide boundedness from below of $m^\epsilon + \epsilon$ independent of $\epsilon$, which implies the boundedness of $-(m^\epsilon + \epsilon)^{-\alpha}$ in $L^\infty(\Tt^d)$. Consequently, integral Bernstein estimates for viscous HJB equations (see \cite{lasry1989nonlinear, PV15}) guarantee that $D u^\epsilon$ is bounded in $L^r(\Tt^d)$ for all $r$. This is enough to have the optimal drift $-D_pH(x,Du_\epsilon)$ in $L^r(\Tt^d)$ for all $r$. Hence, $m^\epsilon$ is in $C^{0,\alpha}$ by classical regularity for the Kolmogorov equation. By bootstrapping in the two equations,  we have enough regularity for $u^\epsilon$ and $m^\epsilon$ to pass to the limit in \eqref{mfgregsta} and obtain smooth solutions.

\end{proof}

\section{Time-dependent problem}\label{sec_td}

\subsection{Estimates for the regularized MFG}\label{hle}
Now, we consider estimates for solutions to Problem \ref{P2}.
We begin with an estimate that is a consequence of the  optimal control representation discussed in the Introduction. Next, we use a modification of a standard technique in MFGs to prove bounds on the coupling $\frac{1}{(m^\epsilon+\epsilon)^\alpha}$
that are uniform in $\epsilon$. These estimates are essential to prove the convergence of solutions as $\epsilon \to 0$. 
\begin{Proposition}
\label{plb}
Assume A\ref{sub1} holds. Let $(u^\ep,m^\ep)$ solve Problem \ref{P2}.   
Then, for any solution, $\zeta:\Tt^d\times (t,T]\to \Rr$, of the heat equation
\begin{equation}
\label{diflaw2}
\zeta_s\,-\,\De \zeta\,=\,0, 
\end{equation}
with $\zeta(x,t)=\zeta_0$, 
we have the following upper bound: 
\begin{align}
\label{lhe2}
\int_{\Tt^d} u^\ep(x,t)\zeta_0(x)\, dx \le
&-\int_t^T\int_{\Tt^d}\frac {\zeta(x,s)\,dx\, ds}{(m(x,s)+\ep)^\af}
+\int_{\Tt^d} u^T(x) \zeta(x,T)\,dx. 
\end{align} 
\end{Proposition}
\begin{Remark}
The solution of the heat equation \eqref{diflaw2} gives the density of 
the solution to the stochastic differential equation \eqref{dyn} with $\bv=0$. This choice of control is sub-optimal for the corresponding control problem and, hence,  the bound in \eqref{lhe2}. \
\end{Remark}
\begin{proof}
First, we multiply
\eqref{diflaw2} by $u$ and multiply the first equation in \eqref{eq:mfgreg} by 
$\zeta$. Next, we subtract these equations and integrate in $\Tt^d$ to conclude that 
\[
\frac{d}{dt}\int_{\Tt^d}u\zeta\, dx=\int_{\Tt^d}\Bigl(H(x,Du) 
+\frac 1{(m+\ep)^\af}\Bigr)\zeta\, dx. 
\]
Because of A\ref{sub1}, $H\ge 0$. Thus, integrating in time, we obtain the result. 
\end{proof}

Natural choices for $\zeta_0$ include the Lebesgue measure,  the measure $m^0$,  or $\zeta_0=\delta_{x_0}$ for $x_0\in \Tt^d$. The latter yields pointwise estimates.

\begin{cor}
Assume A\ref{sub1} holds. Let $(u^\ep,m^\ep)$ solve Problem \ref{P2}.
Then, we have the following two estimates: 
\begin{equation}
\label{etreze}
\int_{\Tt^d}u^\ep(x,0)m^0 \,dx\le
-\int_0^T\int_{\Tt^d}\frac {\mu(x,t) \,dx\, dt}{(m (x,t)+\ep)^\af} 
+\int_{\Tt^d} u^\ep(x,T) \mu(x,T)\,dx,
\end{equation}
where $\mu(x,t)$ is the solution to the heat equation \eqref{diflaw2} with $\mu(x,0)=m^0$, 
and
\begin{equation}
\label{ecatorze}
\int_{\Tt^d} u^\ep(x, 0)\,dx\le -\int_0^T \int_{\Tt^d}
\frac {dx\, dt}{(m (x,t)+\ep)^\af}+\int_{\Tt^d} u^\ep(x,T) \,dx. 
\end{equation}
\end{cor}
\begin{proof}
Both estimates follow from 
Proposition \ref{plb}, by setting $\zeta_0=m^0$ and $\zeta_0=1$.  
\end{proof}

Next, we obtain a first-order estimate for solutions to Problem \ref{P2}. For  $f:\Tt^d\to\Rr^d$, we define
\[
\osc_{\Tt^d} f\,:=\,\sup_{x\in\Tt^d} f(x)\,-\,\inf_{x\in\Tt^d} f(x). 
\]
\begin{Proposition}
\label{pehm}
Assume A\ref{sub1}-A\ref{sub2} hold and $\alpha\neq 1$. Then, there exists a constant $C>0,$ such that for any solution $(u^\ep,m^\ep)$ of Problem \ref{P2}, we have
\begin{equation}
\label{ihm}
\int_0^T \int_{\Tt^d} c H(x, D_xu) m^\ep
+\frac{(m^\ep+\ep)^{1-\af}}{\af-1} \,dx dt 
\le CT+C \osc_{\Tt^d} u^T.
\end{equation}
\end{Proposition}
\begin{proof}
We have
\[
\frac{d}{dt}\int_{\Tt^d}u m \,dx+\int_{\Tt^d}(D_pHD_xu-H)m \,dx 
=\int_{\Tt^d}\frac{m }{(m+\ep)^\af}\,dx.
\]
From the preceding identity,  we get
\begin{align*}
\int_0^T\int_{\Tt^d} (D_pH D_xu-H) m \,dx\,dt=
&\int_0^T\int_{\Tt^d} \frac{m}{(m+\ep)^\af}\,dx\\&+\int_{\Tt^d}\left( u(x,0)
  m(x, 0)- u(x,T) m(x, T)\right)\,dx.  
\end{align*}
Accordingly, using Assumption A\ref{sub2}, 
\begin{align*}
c \int_0^T\int_{\Tt^d} H(x,Du) m \le 
& \int_0^T\int_{\Tt^d} \frac{m}{(m+\ep)^\af}\,dx dt \\
&+\int_{\Tt^d}\left(u(x,0) m(x,0)- u(x,T) m(x,T)\right)\,dx+CT. 
\end{align*}
Now, we use  \eqref{etreze} to conclude
that
\begin{align*}
c \int_0^T\int_{\Tt^d}H(x,Du)m\, dx\,dt \le 
CT&+\int_{\Tt^d} u(x,T)(\mu(x,t)- m(x,T))\,dx\\
&+\int_0^T\int_{\Tt^d} \frac{(m -\mu)}{(m+\ep)^\af}\,dx\,dt.  
\end{align*}
Next, because
\[
        z\mapsto z^{1-\af}/(\af-1)
\]      
is a convex function and $\alpha\neq 1$, we have 
\[
\frac{m-\mu}{(m+\ep)^\af} \le\frac{(\mu+\ep)^{1-\af} -(m+\ep)^{1-\af}}
{\af-1}. \]
Therefore, \begin{align*}
c \int_0^T\int_{\Tt^d} H(x, Du) m &+\int_0^T
\int_{\Tt^d} \frac{(m+\ep)^{1-\af}}{\af-1} \,dx\, dt\\ 
&\
\le CT+\osc u(\cdot , T)+\int_0^T\int_{\Tt^d}
\frac{(\mu +\ep)^{1-\af}}{\af-1} \,dx\, dt.    
\end{align*}
Finally, because $\min m^0\le\mu\le\max m^0$, we have that $(\mu+\ep)^{1-\af}$ is
uniformly bounded and, thus,  \eqref{ihm} follows. 
\end{proof}

\begin{Proposition}\label{prop p1}
 Assume A\ref{sub1}-A\ref{sub2} hold. Then,
there exist constants, $c>0$ and $C>0$, independent on $\epsilon$, such that for any solution $(u^\ep,m^\ep)$ of Problem \ref{P2}, we have
                \[
                \int_0^T\int_{\Tt^d}\frac{1}{(m^\ep+\ep)^\af}\,
+c\int_0^T\int_{\Tt^d}H(x,Du^\ep)m^\ep \,dx\, dt+c \int_0^T\int_{\Tt^d}H(x,Du^\ep)\,dx\, dt\le\,C.
                \]      
\end{Proposition}
\begin{proof}
        For ease of presentation, we drop the superscript $\ep$ in what follows.
        We begin by multiplying the first equation in \eqref{eq:mfgreg} by $(m-m^0)$
        and the second one by $(u^T-u)$. Next, adding the resulting expressions  and
        integrating by parts, we obtain: 
        \begin{align*}
                -\frac{d}{dt}\int_{\Tt^d}\left(u-u^T\right)\left(m-m^0\right)
                \,+&\,\int_{\Tt^d}H(m-m^0)-m\,D_pH\cdot D(u-u^T)\\
&=\,\int_{\Tt^d}u^T\De m\,-\,m^0\De u\,-\,\int_{\Tt^d}\frac{m-m^0}{\left(m+\ep\right)^{\af}}.
        \end{align*}
Integrating in $[0,T]$, we have
\begin{align*}
  \int_0^T\int_{\Tt^d}\frac{m^0+\ep}{(m+\ep)^\af}&
=\int_0^T\int_{\Tt^d}(m+\ep)^{1-\af}+(H-D_pH\cdot Du)m-Hm^0\\
&+\int_0^T\int_{\Tt^d}u^T(-\div(D_pHm)-\De m)+m^0\De u\\
&\le C+\int_{\Tt^d}u^T(m^0-m)-\int_0^T\int_{\Tt^d}cHm+Hm^0+Du\cdot Dm^0.
\end{align*}
The prior inequality follows from  the boundedness  of
$$
        \int\limits_0^T\int\limits_{\Tt^d}(m+\ep)^{1-\af}dx,
$$
which holds,  for $0<\af\le 1$, because $m$ is a probability density,  and, for $\af>1$, due to \eqref{ihm}.
\end{proof}

\subsection{Regularity for the Fokker-Planck equation}\label{rfp}
Next, we  use the structure of the Fokker-Plank equation to improve the  a priori  integrability for the singularity. In what follows, we set
\[
        2^*\,=\,\frac{2d}{d\,-\,2},
\]
the Sobolev conjugated exponent of 2.

\begin{teo}\label{em0}
Assume A\ref{sub1}-A\ref{sub2} hold. There exists $C>0$, independent of $\epsilon$, such that for any solution $(u^\ep,m^\ep)$ to Problem \ref{P2}, we have
\[
\Big\|\frac 1{m^\ep+\ep}\Big\|_{L^\infty([0,T],L^{\frac{\af(2-\ga)}{\ga}}(\Tt^d))}\,+\,
 \int_0^T \Big\|\frac 1{m^\ep+\ep}\Big\|_\frac{2^*\af(2-\ga)}{2\ga}^\frac{\af(2-\ga)}{\ga}\,\leq\, C,
 \]
uniformly in $\epsilon$.
\end{teo}
\begin{proof} As before, we drop the superscript $\ep$. Fix $\beta>0$, arbitrarily. From the second
  equation in \eqref{eq:mfgreg}, we have
  \begin{align}\notag
\frac{d}{dt}\int_{\Tt^d} \frac{dx}{(m+\ep)^\be} &= -\be(\be+1)
\left(\int_{\Tt^d} \frac{D_pH\cdot D_xm}{(m+\ep)^{\be+1}}dx
+\int_{\Tt^d} \frac{|D_xm|^2}{(m+\ep)^{\be+2}} dx\right)\\\notag
&\le \frac{\be(\be+1)}2\left(\int_{\Tt^d} \frac{|D_pH|^2}{(m+\ep)^{\be}}dx
 -\int_{\Tt^d} \frac{|D_xm|^2}{(m+\ep)^{\be+2}} dx\right)\\
&\le \frac{\be(\be+1)}2\int_{\Tt^d} \frac{|D_pH|^2}{(m+\ep)^{\be}}dx 
-\frac{2(\be+1)}{\be}\int_{\Tt^d}|D_x(m+\epsilon)^{\frac{-\be}{2}}|^2dx.\label{fokker}
\end{align}
Using 
Young's inequality, we get\begin{align} 
\frac{\be(\be+1)}2\int_{\Tt^d} \frac{|D_pH|^2}{(m+\ep)^{\be}}dx
&\le C_\de \int_{\Tt^d} |D_pH|^{\ga/(\ga-1)}dx+ \de \int_{\Tt^d}\frac{dx}{(m+\ep)^{\be\ga/(2-\ga)}}.
\label{young}
\end{align}
Integrating \eqref{fokker} in time in $[0,\tau]$, with $0\leq \tau\leq T$ and using \eqref{young}, we obtain
\begin{multline}\label{igm-eq2}
\int_{\Tt^d} \frac{dx}{(m(x,\tau)+\ep)^\be}
+\frac{2(\be+1)}{\be}\int_0^\tau\int_{\Tt^d}|D_x(m+\epsilon)^{\frac{-\be}{2}}|^2dx\,dt\\
\le\int_{\Tt^d} \frac{dx}{(m^0+\ep)^\be(x)}+ C_\de  \int_0^\tau \int_{\Tt^d} |D_pH|^{\ga/(\ga-1)}dx\, dt 
+ \de \int_0^\tau \int_{\Tt^d}\frac{dx\, dt}{(m+\ep)^{\be\ga/(2-\ga)}}.
\end{multline}

By choosing $\be=\af(2-\ga)/\ga$ and using Proposition \ref{prop p1}, we get from \eqref{igm-eq2} that 
\begin{equation}
\label{est}
\int_{\Tt^d} \frac{dx}{(m(x,\tau)+\ep)^{\af(2-\ga)/\ga}}
+\int_0^\tau\int_{\Tt^d}|D_x(m+\epsilon)^{\frac{\af(\ga-2)}{2\ga}}|^2dx\, dt\le C.
\end{equation}

Now, we observe that  Sobolev's inequality yields\[ \Big\|\frac 1{m+\ep}\Big\|_\frac{2^*\be}2^\be
\le C\left(\int_{\Tt^d} \frac{dx}{(m(x,t)+\ep)^\be}+
\int_{\Tt^d}|D_x(m+\epsilon)^{-\frac\be 2}|^2dx\right).  \]
Accordingly, by setting 
\[
        \be=\af(2-\ga)/\ga
\]
and using \eqref{est}, we obtain that 
\[\int_0^T \Big\|\frac 1{m+\ep}\Big\|_\frac{2^*\af(2-\ga)}{2\ga}^\frac{\af(2-\ga)}{\ga}\le C.\]
\end{proof}

\begin{teo}
\label{em}
Assume A\ref{sub1}-A\ref{gamma} hold. Then, for any $r\ge 1$
 there exists $C_r>0,$ such that for any solution $(u^\ep,m^\ep)$ of Problem \ref{P2}, we have\[
\Big\|\dfrac 1{m^\ep+\ep}\Big\|_{L^\infty([0,T],L^r(\Tt^d))}\,\leq\,C_r,
\]
uniformly in $\epsilon$.
\end{teo}
\begin{proof}
First, we define a sequence $(\beta_n)_{n\in\mathbb{N}}$ inductively  as follows: 
\[
        \be_{n+1}=\frac{(2-\ga)}{(\ga-1)d}\be_n,
\]
with\[
        \be_0={\af(2-\ga)}/\ga,
\]
Because Assumption A\ref{gamma} implies that
\[ \frac{ (2-\ga)}{(\ga-1)d}>1, \]

$\beta_{n}\to \infty$ as $n\to \infty$. 

Next, we claim that for any $n\in\mathbb{N,}$ there exists $C_n>0,$ such that
\[
        \int_{\Tt^d} \frac{dx}{(m(x,\tau)+\ep)^{\be_{n}}}\,\leq\,C_n,
\]
uniformly in $\epsilon$. The proof of the claim follows by induction on $n$.
For that, set 
\[
        \lam=\dfrac{2-\ga}\ga.
\]
Then,
\[
        \frac{(d-2)\lam}{d\be_{n+1}}+\frac{1-\lam}{\be_n}=\frac{2-\ga}{\be_{n+1}\ga}. 
\]

Therefore, using interpolation, we get\[
\Big\|\frac 1{m+\ep}\Big\|_{\be_{n+1}\ga/(2-\ga)}\le 
\Big\|\frac 1{m+\ep}\Big\|_{\be_n}^{1-\lam}\Big\|\frac 1{m+\ep}\Big\|_{2^*\be_{n+1}/2}^{\lam}.
\]
The induction hypothesis $\Big\|\dfrac 1{m+\ep}\Big\|_{\be_n}\le C$ implies
\begin{align}
\label{igm-eq5a}
\nonumber\int_{\Tt^d} \frac{dx}{(m+\ep)^{\be_{n+1}\ga/(2-\ga)}}
&\,=\,\Big\|\frac 1{m+\ep}\Big\|_{\be_{n+1}\ga/(2-\ga)}^{\be_{n+1}\ga/(2-\ga)}\\\nonumber
&\le\, C\Big\|\frac 1{m+\ep}\Big\|_{2^*\be_{n+1}/2}^{\lam \be_{n+1}\ga/(2-\ga)}\\&=\,
C\Big\|\frac 1{m+\ep}\Big\|_{2^*\be_{n+1}/2}^{\be_{n+1}}.
\end{align}
A further application of Sobolev's theorem produces
\[ \Big\|\frac 1{m+\ep}\Big\|_{2^*\be_{n+1}/2}^{\be_{n+1}}
\le C\left(\int_{\Tt^d} \frac{dx}{(m(x,t)+\ep)^{\be_{n+1}}}+
\int_{\Tt^d}|D_x(m+\epsilon)^{\frac{-\be_{n+1}}{2}}|^2dx\right).\]
Therefore, for each fixed $t,$ we have 
\begin{align*}
\int_{\Tt^d}\frac{dx}{(m+\ep)^{\be_{n+1}}}\le \left(\int_{\Tt^d}\frac{dx}{(m+\ep)^{\be_{n+1}\ga/(2-\ga)}}\right)^{(2-\ga)/\ga}
&\le C_\zeta+\zeta\int_{\Tt^d}\frac {dx}{(m+\ep)^{\be_{n+1}\ga/(2-\ga)}}\\
&\le C_\zeta+\zeta \Big\|\frac 1{m+\ep}\Big\|_{2^*\be_{n+1}/2}^{\be_{n+1}}.
\end{align*}
Thus,
\begin{equation}\label{igm-eq3a}
 \Big\|\frac 1{m+\ep}\Big\|_{2^*\be_{n+1}/2}^{\be_{n+1}}\le C \int_{\Tt^d}|D_x(m+\epsilon)^{\frac{-\be_{n+1}}{2}}|^2dx+
C_\zeta+\zeta \Big\|\frac 1{m+\ep}\Big\|_{2^*\be_{n+1}/2}^{\be_{n+1}}.
\end{equation}
From \eqref{igm-eq2}, \eqref{igm-eq5a}, and \eqref{igm-eq3a}, taking $\de$ and
$\zeta$ small enough, we have, for some $\de_1>0,$
\begin{align*}
\int_{\Tt^d} \frac{dx}{(m+\ep)^{\be_{n+1}}(x,\tau)}&
+\delta_1 \int_0^\tau\Big\|\frac 1{m+\ep}\Big\|_{2^*\be_{n+1}/2}^{\be_{n+1}}dt\\
&\le  C+\int_{\Tt^d} \frac{dx}{(m^0+\ep)^{\be_{n+1}}(x)}+ C\int_0^\tau\int_{\Tt^d}|D_pH|^{\ga/(\ga-1)}dx\,dt.
\end{align*}
\end{proof}

\subsection{Estimates for the Hamilton-Jacobi equation in $L^\infty(\Tt^d\times [0,T])$}\label{sec:reg-adj}

Now, we obtain uniform bounds for solutions of the
Hamilton-Jacobi equation in \eqref{eq:mfgreg}. \begin{Proposition}
\label{plest}
Assume A\ref{sub1}-A\ref{sub2} hold and let $(u^\ep,m^\ep)$ solve Problem \ref{P2}.
Then,
\begin{equation}
\label{lest}
u^\ep(x,t) \le  \max_{y\in\Tt^d}u^T(y).
\end{equation}
\end{Proposition}
\begin{proof}
The upper bound follows from the maximum principle, taking into account that 
$H\ge 0$ and  that the mean-field coupling is negative. 
\end{proof}

To investigate lower bounds for $u^\epsilon$, we introduce the adjoint equation 
\begin{equation}
\label{ADJ2}
\rho_t-\De \rho -\div(D_pH \rho)=0\;\;\;\;\;\mbox{in}\;\;\;\;\;\Tt^d\times[\tau,T],
\end{equation}
with initial data
\begin{equation}
\label{ADJini2}
\rho(\cdot,\tau)=\delta_{x_0}.
\end{equation}
Combining the adjoint equation and the first equation in \eqref{eq:mfgreg}, we have the following representation formula for $u$:  
\begin{equation}
\label{7A2}
u(x_0,\tau)=\int_\tau^T\int_{\Tt^d}\left(D_pH D_xu-H-\frac 1{(m+\ep)^\af}\right)\rho dxdt+
\int_{\Tt^d}u^T(x)\rho(x,T)dx. 
\end{equation}
\begin{Proposition}
\label{pram}
Assume that A\ref{sub1}-A\ref{sub2} hold. Then, for $p,\,q>1,$ such that
\[
\dfrac 1p+\dfrac 1q=1
\] 
and any solution $(u^\ep,m^\ep)$ to Problem \ref{P2}, we have
\begin{equation}
  \label{umax}
u(x_0,\tau)\ge -C-
\Big\|\frac 1{(m+\ep)^\af}\Big\|_{L^\infty(0,T;L^p(\Tt^d))}
\|\rho\|_{L^1(\tau,T;L^q(\Tt^d))},
\end{equation}
where $\rho$ solves \eqref{ADJ2} with initial data \eqref{ADJini2}.
\end{Proposition}
\begin{proof}
Using A\ref{sub2} in \eqref{7A2}  and combining it with A\ref{sub1}, we get \eqref{umax}.

\end{proof}

\begin{cor}
\label{cor1}
Assume A\ref{sub1}-A\ref{sub2} hold. Then, for $q>1$ and any solution $(u^\ep,m^\ep)$ to Problem \ref{P2}, we have 
\begin{equation}
\label{7A3}
\int_\tau^T\int_{\Tt^d}H\rho\le C+C\Big\|\frac 1{(m+\ep)^\af}\Big\|_{L^\infty(0,T;L^p(\Tt^d))} \|\rho\|_{L^1(\tau,T;L^q(\Tt^d))},
\end{equation}
where $\rho$ solves \eqref{ADJ2} with initial data \eqref{ADJini2}.  
\end{cor}
\begin{proof}
Let $p>1$ be given by $\dfrac 1p+\dfrac 1q=1$.
Using A\ref{sub2} in \eqref{7A2}, we get
\begin{align*}
  c\int_\tau^T\int_{\Tt^d}H\rho&\le C+u(x_0,\tau)+
\Big\|\frac 1{(m+\ep)^\af}\Big\|_{L^\infty(0,T;L^p(\Tt^d))}
\|\rho\|_{L^1(\tau,T;L^q(\Tt^d))}.
\end{align*}
To end the proof, we apply Proposition \ref{plest}.
\end{proof}
\begin{Proposition}
\label{prop2}
Assume that A\ref{sub1}-A\ref{sub2} hold. Then, for $0<\nu<1$ and any solution $(u^\ep,m^\ep)$ to Problem \ref{P2}, we have
\[
        \int_\tau^T\int_{\Tt^d}|D\rho^{\nu/2}|^2dx\,dt\le C+C\|Du^\ep\|^{2(\ga-1)}_{L^{\infty}(0,T;L^{\infty}(\Tt^d))},
\]
where $\rho$ solves \eqref{ADJ2} with initial data \eqref{ADJini2}.
\end{Proposition}
\begin{proof}
Multiply \eqref{ADJ2} by $\nu\rho^{\nu-1}$ and integrate by parts. Then,
\begin{align}\notag
\int_{\Tt^d}(\rho^\nu(x,T)-\rho^\nu(x,\tau))dx\,&
+\nu(\nu-1)\int_\tau^T\int_{\Tt^d}\rho^{\nu-1}D_pH(x,Du)D\rho dx\,dt \\
&=\frac{4(1-\nu)}{\nu}\int_\tau^T\int_{\Tt^d}|D(\rho^{\nu/2})|^2 dx\,dt. 
\label{blabla1}
\end{align}
Because $\rho(\cdot,t)$ is a probability measure and $0<\nu<1$,
\begin{equation}
\label{rae1}
\int_{\Tt^d} \rho^\nu(x,t) dx\le 1.
\end{equation}
Moreover, we have the following estimate
\[\left|\int_\tau^T\int_{\Tt^d}\rho^{\nu-1}D_pH\cdot D\rho dx\, dt\right|
\le \frac 12\int_\tau^T\int_{\Tt^d}|D_pH|^2 dx\,dt+
\frac{2}{\nu^2}\int_\tau^T\int_{\Tt^d}|D(\rho^{\nu/2})|^2 dx\,dt.
\]
Thus,
\[\int_\tau^T\int_{\Tt^d}|D\rho^{\nu/2}|^2dx\,dt\le 
C+C\int_\tau^T\int_{\Tt^d}|D_pH|^2 dx\,dt\le 
C+C\|Du^\ep\|^{2(1-\ga)}_{L^{\infty}(0,T;L^{\infty}(\Tt^d))}.\]
\end{proof}
\begin{Lemma}\label{lem1}
Assume A\ref{sub1}-A\ref{sub2} hold. Also, suppose $\dfrac 2{2^*}<\tilde{\nu}<1$, $0<\kappa\le\tilde{\nu,}$ and let $q>1$ satisfy
\begin{equation}
  \label{eq:def-q}
  \frac{1}{q}=1-\kappa+\frac{2\kappa}{2^*\tilde{\nu}}.
\end{equation}
Then, for any solution $(u^\ep,m^\ep)$ to Problem \ref{P2} and any solution $\rho$ to \eqref{ADJ2} with initial data \eqref{ADJini2}, we have
$$\|\rho\|_{L^1(\tau,T;L^q(\Tt^d))}\le C+C\|Du^\ep\|^{2(\ga-1)}_{L^\infty(\Tt^d\times[0,T])}.$$
\end{Lemma}

\begin{proof}
Using  H\"older's inequality, we have
$$\left(\int_{\Tt^d}\rho^q dx\right)^\frac{1}{q}\le\left(\int_{\Tt^d}\rho dx\right)^{1-\kappa}
\left(\int_{\Tt^d}\rho^\frac{2^*\tilde{\nu}}{2}dx\right)^\frac{2\kappa}{2^*\tilde{\nu}}.$$
By Sobolev's Theorem, we have 
$$\left(\int_{\Tt^d}\rho^\frac{2^*\tilde{\nu}}{2}\right)^\frac{2\kappa}{2^*\tilde{\nu}}\le C
+C\left(\int_{\Tt^d}\big|D\bigl(\rho^\frac{\tilde{\nu}}{2}\bigr)\big|^2dx\right)^\frac{\kappa}{\tilde{\nu}},$$
and, therefore,
$$\int_\tau^T\left(\int_{\Tt^d}\rho^qdx\right)^\frac{1}{q}\le C+
C\int_\tau^T\left(\int_{\Tt^d}\big|D\bigl(\rho^\frac{\tilde{\nu}}{2}\bigr)\big|^2dx\right)^\frac{\kappa}{\tilde{\nu}}.$$
Because $\kappa\le\tilde{\nu}$, the previous computation and 
Proposition \ref{prop2} give that
$$\|\rho\|_{L^1(\tau,T;L^q(\Tt^d))}\le C\,+\,C\|Du^\ep\|_{L^\infty(\Tt^d\times[0,T])}^{2(\ga-1)},$$
which concludes the proof.
\end{proof}

\begin{cor}\label{cor2}
Assume A\ref{sub1}-A\ref{sub2} hold. Let $\dfrac 2{2^*}<\tilde{\nu}<1$ and select $0<\kappa\le\tilde{\nu}$ and $p>1,$ such that 
\begin{equation}
  \label{eq:def-p}
  \frac{1}{p}=\kappa-\frac{2\kappa}{2^*\tilde{\nu}}.
\end{equation} 
Then, for any solution $(u^\ep,m^\ep)$ to Problem \ref{P2} and any solution $\rho$ to \eqref{ADJ2} with initial data \eqref{ADJini2}, we have
$$\int_\tau^T\int_{\Tt^d}H\rho dx\,dt \le 
C+C\Big\|\frac 1{(m+\ep)^\af}\Big\|_{L^\infty(0,T;L^p(\Tt^d))}
\Bigl(1+\|Du^\ep\|_{L^\infty(\Tt^d\times[0,T])}^{2(\ga-1)}\Bigr).$$
\end{cor}

\begin{proof}
The result follows by combining Corollary \ref{cor1} with Lemma \ref{lem1}.
\end{proof}

\subsection{Proof of Theorem \ref{mainteo}}\label{liphj}

The proof of Theorem \ref{mainteo} relies on  Lipschitz bounds for the solutions of the Hamilton-Jacobi equation in \eqref{eq:mfgreg} that are uniform in $\epsilon$. Once we establish the Lipschitz regularity, the proof follows standard arguments in parabolic regularity theory. Thus, next, we combine the estimates from the preceding sections to obtain uniform bounds for $Du^\epsilon$ in $L^\infty(\Tt^d\times[0,T])$. 

\begin{Proposition}\label{prop}
Assume A\ref{sub1}-A\ref{gamma} hold and let $(u^\ep,m^\ep)$ solve Problem
\ref{P2}.
Then,
\[
\left\|Du^\ep\right\|_{L^\infty(\Tt^d\times[0,T])}\le C+
C\|Du^\ep\|^{2(\ga-1)}_{L^\infty(\Tt^d\times[0,T])}.\]
\end{Proposition}
\begin{proof}
For ease of notation, we omit the superscript $\ep$ and set $g=(m +\ep)^{-\af}$. 
By Theorem \ref{em0}, $\|g\|_{L^\infty(0,T;L^p(\Tt^d))}$ is uniformly bounded for any $p>1$.
Let $\rho$ solve \eqref{ADJ2} with initial data \eqref{ADJini2}.
We fix a unit vector $\xi\in\Rr^d$ and differentiate the first equation in 
\eqref{eq:mfgreg} in the $\xi$ direction. 
Next, we multiply the resulting equation by $\rho$ and \eqref{ADJ2} by $u_\xi$. By adding them and 
integrating by parts, we conclude that 
$$u_\xi(x_0,\tau)=\int_\tau^T\int_{\Tt^d}-D_\xi H\rho-g_\xi\rho dx\,dt+
\int_{\Tt^d}(u^T)_\xi\rho(x,T)dx.$$ 
Take $p>d$ and define $\tilde{q}, \tilde{\nu}$ by
\[\tilde{\nu}=\frac 1p+\frac 12+\frac 1{2^*},\quad
\frac{1}{\tilde{q}}+\frac 1p=\frac 12.\]
Thus,  $\tilde{\nu}<\dfrac 1d+\dfrac 12+\dfrac{d-2}{2d}=1$.
Letting $\kappa=\tilde{\nu}\dfrac d{d+p}$, we obtain
\eqref{eq:def-p}. 

Assumption A\ref{sub1} and Corollary \ref{cor2} imply 
\begin{align*}
\left|\int_\tau^T\int_{\Tt^d}-D_\xi H\rho dx\,dt\right|
&\le C+C\int_\tau^T\int_{\Tt^d}H\rho dx\,dt \\
&\le C+C \|g\|_{L^\infty(0,T;L^p(\Tt^d))}\bigl(1+\|Du^\ep\|_{L^\infty(\Tt^d\times[0,T])}^{2(\ga-1)}\bigr).
\end{align*}
Moreover, 
$$\left|\int_{\Tt^d}(u^T)_\xi\rho(x,T)dx\right|\le C,$$
because the preceding expression depends only on the terminal data and $\rho(\cdot, T)$ is a probability density. It remains to bound the term 
$$\int_\tau^T\int_{\Tt^d}g_\xi\rho dx.$$
Integration by parts gives the following estimate:
$$\left|\int_\tau^T\int_{\Tt^d}g_\xi\rho dx\,dt\right|\le 
C\|g\|_{L^\infty(0,T;L^p(\Tt^d))}\|\rho^{1-\frac{\tilde {\nu}}{2}}\|_{L^2(\tau,T;L^{\tilde{q}}(\Tt^d))}
\|D\rho^\frac{\tilde {\nu}}{2}\|_{L^2(\Tt^d\times[\tau,T])}.$$
From Proposition \ref{prop2}, it follows that
$$\|D\rho^\frac{\tilde {\nu}}{2}\|_{L^2(\Tt^d\times[\tau,T])}
\le C+C\|Du\|_{L^\infty(\Tt^d\times[0,T])}^{\ga-1}.$$ 
Moreover, setting $\theta=\dfrac{\tilde{\nu}}{2-\tilde {\nu}}$, we have
\begin{equation}\label{rest7}
\frac{1}{\tilde{q}\left(\frac{2-\tilde{\nu}}{2}\right)}=
1-\theta+\frac{2\theta}{2^*\tilde{\nu}}.
\end{equation}
Then
$$\left(\int_{\Tt^d}\rho^{\tilde{q}\left(\frac{2-\tilde{\nu}}{2}\right)}\right)^\frac{1}{\tilde{q}
\left(\frac{2-\tilde{\nu}}{2}\right)}\le \left(\int_{\Tt^d}\rho\right)^{1-\theta}
\left(\int_{\Tt^d}\rho^\frac{2^*\tilde{\nu}}{2}dx\right)^\frac{2\theta}{2^*\tilde{\nu}}.$$
Next, Sobolev's Theorem yields 
$$\left(\int_{\Tt^d}\rho^{\tilde{q}\left(\frac{2-\tilde{\nu}}{2}\right)}\right)^\frac{1}{\tilde{q}
\left(\frac{2-\tilde{\nu}}{2}\right)}\le C+C\left(\int_{\Tt^d}
\big|D\bigl(\rho^\frac{\tilde{\nu}}{2}\bigr)\big|^2\right)^\frac{\theta}{\tilde{\nu}}.$$
Consequently, 
$$\left(\int_{\Tt^d}\rho^{\tilde{q}\left(\frac{2-\tilde{\nu}}{2}\right)}\right)^\frac{2}{\tilde{q}}
\le C+C\left(\int_{\Tt^d}\big|D\bigl(\rho^\frac{\tilde{\nu}}{2}\bigl)\big|^2\right)
^\frac{(2-\tilde{\nu})\theta}{\tilde{\nu}}.$$
Taking into account that $\frac{(2-\tilde{\nu})\theta}{\tilde{\nu}}=1$ and using Proposition \ref{prop2}, we obtain the estimate
$$\|\rho^{1-\frac{\tilde{\nu}}{2}}\|_{L^2(\tau,T;L^{\tilde{q}}(\Tt^d))}\le C
+C\|Du\|_{L^\infty(\Tt^d\times[0,T])}^{\ga-1}.$$
The previous computation implies \[
\left|u_\xi(x,\tau)\right|\le C+C\|g\|_{L^\infty(0,T;L^p(\Tt^d))}
\bigl(1+\|Du\|^{2(\ga-1)}_{L^\infty(\Tt^d\times[0,T])}\bigr).
\]
Accordingly, the result follows from Theorem \ref{em}.
\end{proof}
\begin{cor}
Assume A\ref{sub1}-A\ref{gamma} hold. Then, there exists a constant, $C>0,$ that does not depend on $\epsilon, $ such that for any solution $(u^\ep,m^\ep)$ to Problem \ref{P2}, we have
\[
        \left\|Du^\ep\right\|_{L^\infty(\Tt^d\times[0,T])}\leq C.
\]
\end{cor}

Finally, we prove our main result in the time-dependent case, Theorem \ref{mainteo}.

\begin{proof}[Proof of Theorem \ref{mainteo}]
Under the assumptions of the theorem,
the existence of smooth solutions to Problem \ref{P2}
follows from a standard argument, see, for example, \cite{GPM2}
or \cite{GPM3}.
In particular, strict convexity (see A\ref{sub1}) is essential for the uniqueness of a solution by the monotonicity technique -- see \cite{ll3}.  
Since $\|(m^\ep+\ep)^{-\af}\|_p$ is uniformly bounded for any
  $p>1$ and $\|Du^\ep\|_\infty$ is uniformly bounded, regularity
  theory applied to the first equation of \eqref{eq:mfgreg} implies that $\|u^\ep_t\|_p$ and
$\|D^2u^\ep\|_p$ are uniformly bounded for any $p>1$. Then, Morrey's inequality implies
that for $0<\be<1$, $\|u^\ep\|_{C^{0,\be}}$  is uniformly bounded.
As in \cite{GPM2}, we get that, for some $0<\be<1$,
$\|m^\ep\|_{C^{0,\be}}$ is uniformly bounded.

Next, using the Hopf-Cole transformation, $v^\ep=\ln(m^\ep+\ep),$
we have
\[v^\ep_t-D_pH\cdot Dv^\ep-\div(D_pH)=|Dv^\ep|^2+\De v^\ep.\]
Consequently,  as in \cite{GPM2}, we obtain that $v^\ep$ is uniformly bounded from below
and uniformly Lipschitz. Thus, $m^\ep$ is also uniformly bounded
away from zero and $(m^\ep)^{-\af}$ is uniformly Lipschitz.
Furthermore, through some subsequence, we have that $u^\ep\to u$, $m^\ep\to m$ in 
$C^{0,\be}(\Tt^d\times[0,T])$ as $\ep\to 0$. Hence,  $u$ is a
(viscosity)  solution of the first equation of \eqref{eq:mfg}. 
Furthermore, $m$ is a weak solution of the second equation in \eqref{eq:mfg}. 
As in \cite{GPM2}, we get uniform bounds in every Sobolev space for $(u^\ep,m^\ep)$.
Hence, $(u,m)$ satisfies the same estimates and, thus,  it is a classical solution. 
\end{proof}

\bibliography{mfg}
\bibliographystyle{plain}
\end{document}